\documentclass[12pt]{amsart}
 \usepackage{latexsym,amsmath,amsfonts,amssymb, amsthm, xcolor, colortbl}      
\newtheorem{thm}{Theorem}[section]
\newtheorem{lemma}[thm]{Lemma}
\newtheorem{proposition}[thm]{Proposition}

\theoremstyle{definition}
\newtheorem{defn}[thm]{Definition}
\newtheorem{conj}[thm]{Conjecture}
\newtheorem{rmk}[thm]{Remark}
\numberwithin{equation}{section}

\begin{document}

\title[Generalized Ramanujan equation]{ 
On the generalized Ramanujan equation $x^2+(2k-1)^y =k^z$ }

\author{Maohua Le}
\address{Institute of Mathematics, Lingnan Normal College, Zhangjiang, Guangdong, 524048, China}
\email{lemaohua2008@163.com}
\author[A. Srinivasan]{Anitha Srinivasan}
\address{Departamento de m\'etodos cuantitativos, Universidad Pontificia de Comillas (ICADE),
C/ Alberto Aguilera, 23 - 28015 Madrid
} 
\email{asrinivasan@icade.comillas.edu}

\begin{abstract}
 A conjecture of N. Terai states that for any integer $k>1$, the equation $x^2+(2k-1)^y =k^z$ has only one solution, namely, $(x, y, z) = (k-1, 1, 2).$ 
  Using the structure of class groups of  binary 
quadratic forms, we prove the conjecture when $4\Vert k$, with $2k-1$ a prime power 
and $4\le k\le 1000$. 
\end{abstract}
\maketitle
\noindent
 {\footnotesize{\textbf{ AMS Subject
Classification}}: 11D09, 11R29, 11E16.
  }

\section{ Introduction}
N. Terai in \cite{T} stated the following conjecture for a particular case of the generalized Ramanujan Nagell equation.
\begin{conj}
For any integer $k>1$, the diophantine equation
\begin{equation}\label{RN}
x^2 +(2k-1)^y =k^z, x, y, z\in \mathbb{N}
\end{equation}
has only the solution $(x, y, z) = (k-1, 1, 2).$
\end{conj} 
The above conjecture has been proved only in some special cases. In \cite{DGX}, M.-J. Deng, J. Guo and A.-J. Xu verified Conjecture 1.1 when $k\equiv 3 \pmod 4$ with $3\le k\le 499.$ 

N. Terai \cite{T} dealt with $k\le 30$ using methods that did not apply to $k= 12, 24$. Subsequently M. A. Bennett and N. Billerey \cite{BB} used the modular approach to solve these two cases. 

Recently Mutlu, Le and Soydan proved the following result.
\begin{thm}[\cite{MLS}]
Under the assumption of the GRH, if $k\equiv 0\pmod 4$ and
$20<k<724$ with $2k-1$ a prime power, then Conjecture 1.1 is true.
\end{thm}
The authors in the theorem above reduce the problem to a Frey curve to which they employ the Cremona elliptic curve database to study the various cases that arise. The upper bound of $k=724$ in their theorem is due 
to the current upper bound of $500000$ for conductors in this data base, which is exceeded for $k=724$. We would also like to remark that these authors  deal with the case of $y\ge 7$, and $y=3$ and $5$ separately, showing that the conjecture does not hold in these cases. 
However it appears that  they have not dealt with the 
case of $y=1$, and hence in our view Conjecture 1.1 is still open for the values of $k$ given 
in the theorem.

We present a method to prove the conjecture for $k$ satisfying the conditions in the theorem above, assuming 
additionally that $4$ exactly divides $k$ (we do not assume the GRH). 
The proof of our theorem begins with a simple observation  that if $(\alpha,\beta,c)$ is a solution of 
(\ref{RN}) and 
$d$ divides $ab$ where $k-1=ab^2$, then 
$d$ divides $\alpha$ (so $\alpha=dx$), and hence  we may rewrite this equation as 
\begin{equation}\label{d1}
d^2x^2+(2k-1)^{\beta}=k^c.
\end{equation} It is known that $\beta$ and $c>1$ are odd in the equation above when 
$4|k$ and $2k-1$ is a prime power.  Keeping this in mind, we may view equation (\ref{d1}) as
a representation of $k^c$ by the binary quadratic form $d^2x^2+(2k-1)y^2$ (note that $\beta$ is odd). Using the structure of the 
class group (of discriminant $-4d^2(2k-1)$), for a given $k$, we are able to find a divisor $d$ such that this form represents only even  powers of $k$.  This leads to a contradiction, since $c$ is odd as mentioned above. 
In our main theorem below we consider only values of $k\le 1000$ merely to exhibit some concrete values, especially since in the current literature the conjecture is open for these values of $k$. 
Indeed, we note that using our method we can show that Conjecture 1.1 is true for greater values of $k$ such as $60040, 40000936$, etc. Nevertheless, it should be noted that while for any given fixed $k$ (size restricted only by computational capacity), we are able to find this divisor $d$, in general we do not know if this integer exists. In Section 4 we present a conjecture which claims that this integer exists for all $k$ such that $4\Vert k$ and $2k-1$ is prime. It should be noted that in the case when $k$ is divisible by $8$, this integer may not exist (for instance for $k=24$ it does not), and hence the exclusion of these values from our theorem.

\begin{thm} If
$4\Vert k$ with $4\le k\le 1000$ and
$2k-1$ is a prime power, then Conjecture 1.1 is true.
\end{thm}
\section{Binary quadratic forms }
 In this section we present the basic theory of
binary quadratic forms. An excellent 
reference is \cite{Ri}, in 
Sections 4 to 7 and 
11 of Chapter 6. 

A {\sl primitive binary quadratic form} $F=(a, b, c)$ of discriminant $\Delta$ is
a function $F(x, y)=ax^2+bxy+cy^2$, where $a, b,c $ are integers
with $b^2-4 a c=\Delta$
and $\gcd(a, b, c)=1$. 
Note that the integers 
$b$ and $\Delta$ have the same parity. All forms considered here
are primitive binary quadratic forms and henceforth we shall
refer to them simply as forms. 

Two forms $F$ and $F'$ are said to be (properly) {\it equivalent}, written as
$F\sim F'$,  if for some
$A=\begin{pmatrix} \alpha &\beta \\ \gamma & \delta \end{pmatrix}
\in SL_2(\mathbb Z)$ (called a transformation matrix),  we have
$F'(x,y)=f(\alpha x+\beta y, \gamma x+\delta y)
=(a',b',c')$, where 
$a', b', c'$ are given by
\begin{equation}
a'=F(\alpha, \gamma),\hskip2mm
b'=2(a\alpha\beta+c\gamma\delta)+b(\alpha\delta+\beta\gamma),\hskip2mm
c'=f(\beta, \delta).
\end{equation}
 It is easy to see
that $\sim$ is an equivalence relation on the set of forms of
discriminant $\Delta$. The equivalence classes form an abelian group
called the  {\it class group} with group law given by composition of
forms. 
The {\it identity form} is defined as the form $(1,0,\frac{-\Delta}{4})$
or $(1, 1, \frac{1-\Delta}{4})$, depending on whether $\Delta$ is even or odd
respectively. 
 The {\it inverse} of
$F=(a, b, c)$ denoted by $F^{-1}$, is given by $(a,-b,c).$
In the following definition we present the formula for composition of forms.

 Let $F_1=(a_1, b_1, c_1) \text{ and }
F_2=(a_2, b_2, c_2)$ be two  binary quadratic forms
of discriminant $\Delta$.
\begin{defn}[Composition] 
Let $l=gcd(a_1, a_2, (b_1+b_2)/2)$ and let $v_1, v_2, w $ be integers
such that 
\begin{equation}\label{gcd}
v_1a_1+v_2a_2+w(b_1+b_2)/2=l.
\end{equation}
 If we define $a_3$ and
$b_3$ as
\begin{equation}\label{a3}
a_3=\frac{a_1a_2}{l^2}
\end{equation}
and
\begin{equation}\label{b3}
b_3 =b_2+2\,\frac{a_2}{l}\,\left(\frac{b_1-b_2}{2}\,\,v_2-c_2w\right) ,
\end{equation}
then  the composition of the forms $(a_1, b_1, c_1)$ and
$(a_2, b_2, c_2)$ is the form $(a_3, b_3, c_3)$, where $c_3$ is
computed using the discriminant equation.
\end{defn}
Note that this gives the multiplication in the class group.

A form $F$ is said to represent an integer $N$ if there exist 
integers $x$ and $y$ such that $F(x,y)=N$. If $\gcd(x, y)=1$, we 
call the representation a primitive one.
Observe that equivalent forms primitively represent the same set
of integers, as do a form and its inverse. Note also that
 if $F$ and $G$ 
are in the identity class, then so are $F^{-1}$ and $FG$.  

We end this section with two elementary 
observations about forms. Firstly,  
 a form $F$ represents primitively an integer $N$ if and only if 
$F\sim (N, b, c)$ for some integers $b, c$. This follows simply 
by noting that 
$F(\alpha, \gamma)=N$ with $\gcd(\alpha, \gamma)=1$ if and only if there 
exists a transformation matrix $A$ as given above such that 
(2.1) holds. Secondly, if 
$b\equiv b'\pmod {2N}$, then the forms 
$(N, b, c)$ and $(N, b',c')$ are equivalent. This equivalence
follows using the transformation matrix
$A=\begin{pmatrix} 1 & \delta \\ 0 & 1 \end{pmatrix}$
where $b'=b+2N\delta$. 

\section{A proposition}

Let $k$ be a positive integer such that 
$4|k$. Let $D$ be an odd positive integer such that $\gcd(D,k)=1$. In this section we present a proposition that is the main tool in our proof. Our first lemma in this
section follows from elementary number theory. Let $w(k)$ denote the number of distinct prime divisors in $k$.
\begin{lemma} Let $k\equiv 0\pmod 4$.
If the congruence 
\begin{equation}\label{4k}
l^2\equiv -D\pmod{4 k}, 0<l<2k, \gcd\left(4k, 2l, \frac{l^2+D}{4 k}\right)=1
\end{equation}
has solutions $l$, then it has exactly $r=2^{w(k)-1}$ solutions.
\end{lemma}
\begin{rmk}\label{4kform}
Observe that if $x$ satisfies $x^2\equiv -D\pmod{4 k}$, then so does 
$2k-x$. However only one of these two solutions satisfies the gcd condition given in (3.1). Therefore only
one leads to a binary quadratic form of discriminant $-4D$, as described below.
\end{rmk}
Let the $r$ solutions of equation (\ref{4k}) be $l_i$ for $i=1,\cdots l_r$. Each such solution corresponds to a binary quadratic form of discriminant $-4D$, namely 
\begin{equation}\label{fi}
f_i=\left(4k, 2l_i, \frac{l_i^2+D}{4 k}\right).
\end{equation}  Using the composition algorithm we can show that for any positive integer $j>1$, we have 
$f_i^j=(4k^j, 2M_j, C_j)$, where $M_j\equiv 2k+l_i \pmod{4k}$. 
 Let us verfiy this by performing the composition $f_i^2$. For notational simplicity,
let $f=f_i=(4k, 2l, c)$. To apply the composition algorithm, we first note that the required gcd in Definition 2.1, namely, 
$\gcd(4k, 4k, 2l)=2$, as $l$ is odd ($l^2\equiv D\pmod 4$). It follows that if $f^2=(a_3, 2l_3,c_3)$, then 
$a_3=4k^2$ (from (\ref{a3})). Let $v_1, v_2, w$ be the 
integers as given in (\ref{gcd}). Then 
as $cw$ is an odd integer, we have from (\ref{b3}) that    
$l_3=l+2k(-cw)\equiv 2k+l\pmod {4k}$. In an identical manner we can compose
$f$ and $f^2$ to get $f^3=(4k^3, 2M, C)$ with $M\equiv 2k+l\pmod {4k}$. Therefore we have shown that
if $j>1$ then 
$f_i^j$ is given by 
\begin{equation}\label{e}
f_i^{j}=(4k^{j}, 2L_i, C_i), \hskip2mm L_i\equiv 2k+l_i \pmod{4k}.
\end{equation}

\begin{proposition}Let $k\equiv 0\pmod 4$ and  $F$ be a form of discriminant $-4D$. Suppose that $z>1$ is a positive integer such that $F$ represents 
$k^z$. Let $n_i$ be the order of $f_i$ (as given in \ref{fi}). Assume that $n_i>1$ for each $i$ and let 
$M=max\{n_i: 1\le i\le r\}$. Then  
 the following are true.
\begin{enumerate}
\item 
For some  $1\le i\le r$, there exists  a positive integer $z_0\le n_i+1$ such that 
$z=z_0+n_i t$  where $t\ge 0$ is an integer.
\item Let $2\le  m_1, m_2, \cdots m_s$   be all 
the exponents less than or equal to $M+1$ such that 
$F$ represents $k^{m_i}$. If all the $n_i$ and $m_i$ are even, then $F$ represents only
even powers of $k$.
\end{enumerate}
\end{proposition}
\begin{proof} Observe that as $F$ represents $k^z$, there is a form $G=(k^z, 2L, C)$
that is equivalent to $F$ (see the last paragraph of Section 2).  It follows from the 
discriminant equation that $L^2\equiv -D \pmod{k^z}$ and hence we have from (3.1) and Remark \ref{4kform}, that
either $L\equiv \pm l_i\pmod{4k}$ or $L\equiv \pm (2k-l_i)\pmod{4k}$.  If $L\equiv \pm l_i\pmod{4k}$, then 
from (\ref{e}) we have $L\equiv \pm L_i\pmod{2k}$. We get the same conclusion in the case when $L\equiv \pm (2k-l_i)\pmod{4k}$. Let us assume that 
\begin{equation}\label{LL}
L\equiv L_i\pmod{2k}
\end{equation}
for some $1\le i\le r$.  We may suppose that $z>n_i+1$ as else we may take $z_0=z$ and $t=0$ in the theorem.

From (\ref{e}) we have $f_i^{n_i}=(4k^{n_i}, 2L_i, C_i)$ (as $n_i>1$). Suppose 
$f=(4k^{n_i}, -2L_i, C_i)$.
Note that $f$ is equivalent to the identity form. Consider the composition of
$G$ with $f$ that yields a form equivalent to $G$. To apply the composition algorithm, we first evaluate
the gcd given in Definition 2.1, namely, $\gcd(k^z, 4k^{n_i}, L-L_i)$. By our assumption above $z>n_i$ and hence $\gcd(k^z, 4k^{n_i})=4k^{n_i}$ 
(as $4|k$). Since $D$ is odd, both $L$ and $L_i$ are odd. Moreover from the discriminant equation we have $L^2-L_i^2=4k^{n_i}\left(-C_i+\frac{Ck^z}{4k^{n_i}}\right)$ 
which combined with equation (\ref{LL})  above gives
$$\gcd(k^z, 4k^{n_i}, L-L_i)=2k^{n_i},$$ 
as $C$ and $C_i$ are odd and $z>n_i+1$.
 From (\ref{a3}) in the composition algorithm if 
$Gf=(a_3,b_3,c_3)$, we have $a_3=k^{z-n_i}$ and thus $F$ represents $k^{z-n_i}$.  Therefore there exists a form $G_1=(k^{z-n_i}, 2L',C')$ equivalent to $F$. Proceeding in an identical fashion as above (using $G_1$ instead of $G$), we continue the process until we obtain that $F$ represents $k^{z_0}$ with $z_0\le n_i+1$
 and thus $z=z_0+n_i t$ for some integer $t\ge 0$ concluding  the proof of the theorem. In the case when  
$L\equiv -L_i\pmod{2k}$ we would proceed in an identical fashion, with the difference that here
we use $f=(4k^{n_i}, 2L_i, C_i)$. 

Note that part (2) of the proposition follows immediately from part (1).
\end{proof}
\section{Proof of Theorem 1.3}

\begin{lemma}\label{ab} Let $x,y,z$ be positive integers such that $x^2+(2k-1)^y=k^z$. If
$k-1=ab^2$ with $a$ square-free, then $x\equiv 0 \pmod{ab}$.
\end{lemma}
\begin{proof}
The lemma follows on looking at the given equation modulo $k-1$ which gives
$x^2\equiv 0 \pmod {ab^2}$.
\end{proof}
\begin{lemma}\label{yz}\cite[Lemma 2.6]{DGX} Let $k\equiv 0\pmod 4$ and
$2k-1$ be a prime power. Then if $x,y,z>2$ is a solution of $x^2+(2k-1)^y=k^z$  then
$y$ and $z$ are odd.
\end{lemma}
\begin{proof}
If $k\equiv 0\pmod 4$ it is easy to see (considering the equation mod $4$) that $y$ is
odd.  In the case when $2k-1$ is a prime power it is shown in \cite[Lemma 2.6]{DGX} that $z$ is odd.
\end{proof}
\begin{lemma}\label{d}Let $k\equiv 0\pmod 4$ and
$2k-1$ be a prime power. Suppose that $k-1=ab^2$ with $a$ square-free. Let $d|ab$ be such that $d>1$ and the form $(d^2, 0, 2k-1)$ represents only even powers of $k$. Then Conjecture 1.1 is true.

\end{lemma}
\begin{proof}

Suppose that $(x,y,z)$ is a solution to equation (\ref{RN}) with $z>2$.
From Lemma \ref{ab} we have  $d|x$ and hence
for $x=d\alpha$, we have 
\begin{equation}\label{alp}
d^2\alpha^2+(2k-1)^y=k^z.
\end{equation}
Note that $y$ is odd by Lemma \ref{yz} and thus the form $F=(d^2, 0, 2k-1)$ 
represents $k^z$ (note from (\ref{alp}) that $F(\alpha, (2k-1)^{(y-1)/2})=k^z$). By the same lemma we have $z$ is odd, which contradicts the hypothesis that $F$ represents only even powers of $k$.
 \end{proof}

\noindent{\bf Proof of Theorem 1.3}

Let $k-1=ab^2$ where $a$ is square-free. For each given $k$, we find a divisor $d|ab$ such that $d>1$ and the form $F=(d^2, 0, 2k-1)$ represents only even powers of $k$, and hence by Lemma \ref{d}, we may conclude that Conjecture 1.1 is true. To show that $F$
represents only even powers of $k$, we use Proposition 3.3 (with 
$D=-d^2(2k-1)$). 
Observe that the identity 
form $(1, 0, (2k-1)d^2)$, does not represent $4k$ (as $d>1$), and hence $f_i$ (that represents $4k$) is not equivalent to the identity. It follows that 
its order in the class group, $n_i>1$ for all $i$. Next, we determine the forms $f_i$ (given in \ref{fi}) and their orders $n_i$ and observe that each $n_i$ is even. Furthermore, 
we calculate all the exponents of $k$ less than or equal to $M+1$ represented by $F$, and observe
that they are also all even. Therefore
Proposition 3.3 (2) holds and we have shown that $F$
represents only even powers of $k$. In Table we present the values of $k$  given in Theorem, along with the corresponding values of $d$, the orders of $f_i$ and the powers $z$ such that $F$ represents $k^z$.
\hfill $\qed$

The proof  of our main theorem is achieved by finding a divisor $d>1$ of $ab$ such that the form 
$F=(d^2, 0, 2k-1)$ represents only even powers of $k$. Based on our computations, we believe that $d=ab$ satisfies this condition for each such $k$ (with $4\Vert k$ and $2k-1$ prime) leading to the following conjecture.
\begin{conj}
Let $k$ be a positive integer such that $4\Vert k$ and $2k-1$ is a prime power. Suppose that 
$k-1=ab^2$ where $a$ is square-free. Then the 
form $(a^2b^2,0,2k-1)$ represents only  even powers of $k$.
\end{conj}



\begin{thebibliography}{99}
\bibitem{BB} {\textsc M. Bennett, N. Billerey,} {\it Sums of two S-units via Frey-Hellegouarch curves,} Math. Comp. 305 (2017) 1375--1401.
\bibitem{DGX}{\textsc M.-J. Deng, J. Guo, A.-J. Xu,}
{\it A note on the Diophantine equation $x^2 + (2c-1)^m=c^n$ },  Bull. Aust. Math.
Soc. {\bf{98}} (2018) 188--195.
\bibitem{MLS}{\textsc Elif Kızıldere Mutlu, Maohua Le, G\"{o}khan Soydan},
{\it A modular approach to the generalized Ramanujan-Nagell equation},
Indagationes Mathematicae,
{\bf 33}(5), September 2022, 992--1000.
\bibitem{Ri}{\textsc P. Ribenboim}, \textit{My Numbers, My Friends, Popular
Lectures on Number Theory}, { Springer-Verlag}, 2000.
\bibitem{T} {\textsc N. Terai}, {\it A note on the diophantine equation $x^2+q^m=c^n$}, Bull. Aust. Math. Soc. {\bf 90} (2014) 20--27.


\end{thebibliography}
\end{document}